\documentclass[9pt]{article}
\usepackage{fullpage,amsmath,amsfonts,amsthm,graphicx,amssymb,textcomp,enumerate,multicol}
\usepackage{hyperref}

\newtheorem{theorem}{Theorem}[section]
\newtheorem{corollary}[theorem]{Corollary}
\newtheorem{lemma}[theorem]{Lemma}

\theoremstyle{definition}
\newtheorem{definition}[theorem]{Definition}

\theoremstyle{remark}
\newtheorem{example}[theorem]{Example}

\theoremstyle{remark}
\newtheorem{remark}[theorem]{Remark}

\theoremstyle{remark}
\newtheorem{question}[theorem]{Question}

\newcommand{\R}{\mathbb{R}}
\newcommand{\Z}{\mathbb{Z}}
\newcommand{\N}{\mathbb{N}}
\newcommand{\T}{\mathbf{T}}

\newcommand{\inlinegfx}[1]{\includegraphics[scale=0.18]{#1}}

\title{Stabilization of the Khovanov Homotopy Type of Torus Links}
\author{Michael Willis \\
Department of Mathematics, University of Virginia\\
\href{mailto:msw3ka@virginia.edu}{\texttt{msw3ka@virginia.edu}}}

\begin{document}

\maketitle

\begin{abstract}
The structure of the Khovanov homology of $(n,m)$ torus links has been extensively studied.  In particular, Marko Sto\v{s}i\'{c} proved that the homology groups stabilize as $m\rightarrow\infty$.  We show that the Khovanov homotopy types of $(n,m)$ torus links, as constructed by Robert Lipshitz and Sucharit Sarkar, also become stably homotopy equivalent as $m\rightarrow\infty$.  We provide an explicit bound on values of $m$ beyond which the stabilization begins.  As an application, we give new examples of torus links with non-trivial $Sq^2$ action.
\end{abstract}

\section{Introduction}
In \cite{Khov} Mikhail Khovanov introduced the Khovanov homology of a knot or link $K$, which categorifies the Jones polynomial of $K$.  Since then, the Khovanov homology of torus links $T(n,m)$ in particular has been widely studied, and many interesting structural results have been found.  In \cite{Sto}, Marko Sto\v{s}i\'{c} used a long exact sequence to prove that the Khovanov homology groups of the $T(n,m)$ stabilize as $m\rightarrow\infty$.  A generalization of this result to the context of tangles came in the form of \cite{Roz}, where Lev Rozansky showed that the Khovanov chain complexes for torus braids also stabilize (up to chain homotopy) in a suitable sense to categorify the Jones-Wenzl projectors.  (At roughly the same time, Benjamin Cooper and Slava Krushkal gave an alternative construction for the categorified projectors in \cite{CK}.)  These results, along with connections between Khovanov homology, HOMFLYPT homology, Khovanov-Rozansky homology, and the representation theory of the rational Cherednik algebra (see \cite{GORS}) have led to conjectures about the structure of the stable Khovanov homology groups in the limit $Kh(T(n,\infty))$ (see \cite{GOR}, and results along these lines in \cite{Hog}).

More recently, in \cite{LS}, Robert Lipshitz and Sucharit Sarkar introduced the Khovanov homotopy type of a knot or link $K$.  This is a link invariant taking the form of a spectrum whose reduced cohomology is the Khovanov homology of $K$.  In a subsequent paper \cite{LLS}, Lipshitz, Sarkar, and Lawson pose the question of stabilization for torus links: given a fixed number of strands $n$, does the Khovanov homotopy type of the torus links $T(n,m)$ stabilize as $m\rightarrow\infty$?

In this paper, we give an affirmative answer to this question, and provide an explicit bound for $m$ beyond which specific wedge summands (which depend on $q$-degree) of the Khovanov homotopy type stabilize.  More precisely, we prove the following theorem and corollaries.

\begin{theorem}
\label{finalresult}
Fix $n\in\N$, and let $\T^m$ denote the $(n,m)$ torus link (viewed as the closure of the $(n,m)$ torus braid on $n$ strands).  For $j\in\Z$ let $\chi^j(K)$ denote the Khovanov homotopy type, in $q$-degree $j$, of a knot or link $K$.  Then for any fixed $q$-degree $a\in\Z$, there exists $M\in\Z$ such that for any $m\geq M$, $\chi^{a+k(n-1)}(\T^{m+k})\simeq\chi^{a}(\T^m)$ for all $k\in\N$, where $\simeq$ denotes stable homotopy equivalence.
\end{theorem}

\begin{corollary}
\label{finalresultcor}
For any fixed $n\in\N$ as above, there exists a well-defined limiting Khovanov homotopy type $\chi(\T^\infty)$.
\end{corollary}

\begin{corollary}
\label{infmanystsquares}
Fixing $n=3$, we have that $\chi(\T^m)$ admits a non-trivial action of the Steenrod Square $\text{Sq}^2$ for all $m\geq 3$.
\end{corollary}

In short, the proof of Theorem \ref{finalresult} proceeds as follows.  In passing from $\T^m$ to $\T^{m+1}$, $n-1$ new positive crossings are introduced, and the normalization of the $q$-degree is shifted as well.  We resolve the new crossings one at a time.  Each time we do so, we can view the chain complex being split into an upward closed and downward closed piece, in the sense of Lipshitz's and Sarkar's Lemma 3.32 in \cite{LS}, recast as Lemma \ref{Lemma3.32} in this paper.  Having fixed our original $q$-degree, we find a bound on $m$ that forces the upward closed pieces (corresponding to the 1-resolutions of crossings) to be acyclic, while simultaneously respecting the shifts in $q$-degree both as we resolve crossings, and as we continue the passage from $\T^{m+1}$ to $\T^{m+2}$ and so forth.  Lemma \ref{Lemma3.32} will allow us to collapse the acyclic pieces, and this will give us the result.

\begin{remark}
\label{StosicDifference}
This proof is similar in spirit to Sto\v{s}i\'{c}'s proof of the stabilization of homology in \cite{Sto}, and will use some similar notations.  However, there is an important difference.  Sto\v{s}i\'{c}'s proof provides stabilization of homology in a fixed homological degree (allowing the $q$-degree to remain free).  On the other hand, our proof will require stabilization of the full chain complex based on a fixed $q$ degree (allowing the homological degree to remain free).  This will in turn require more careful bounding on several parameters.
\end{remark}

\begin{question}
\label{GORSquestion}
As alluded to above, Gorsky, Oblomkov, and Rasmussen have conjectured in \cite{GOR} that the limiting Khovanov homology of $\T^\infty$ for fixed $n\in\N$ is dual to the homology of the differential graded algebra generated by even variables $x_0,\dots,x_n$ and odd variables $\xi_0,\dots,\xi_n$ with respect to a differential $d$ of the form
\[d(\xi_k)=\Sigma_{i=0}^k x_ix_{k-i}.\]
See \cite{GOR} for more details on the homological and $q$ gradings of the generators.  Given Corollary \ref{finalresultcor}, it seems natural to ask whether there exists some similarly elegant description of the limiting spectrum $\chi(\T^\infty)$.
\end{question}

This paper is arranged as follows.  In section 2 we review some of the main properties of the Khovanov homotopy type and end with the key lemma alluded to above that allows us to collapse specific acyclic subcomplexes of our Khovanov chain complexes.  In section 3, we introduce the necessary notations and definitions, placing bounds on the various quantities involved, before outlining the main strategy and giving an example.  In section 4, we use the bounds established in section 3 to prove the main result for $n\geq 4$.  The cases $n=2$ and $n=3$ are markedly simpler than $n\geq 4$, and are discussed in sections 5 and 6.  Section 7 will then use these results to define the Khovanov homotopy type of the torus links in the limit as $m\rightarrow\infty$, giving explicit formulae to calculate the resulting spectra.

\subsection{Acknowledgements}
The author would like to thank Robert Lipshitz for many illuminating and encouraging discussions about the Khovanov homotopy type, as well as Nick Kuhn for several helpful conversations on spectra.  The author would also like to send special thanks to his advisor, Slava Krushkal, for his continued advice, encouragement, and guidance in both the pursuit of this result and the preparation of this paper.

\section{The Khovanov homotopy type}

\subsection{Basic Definition and Properties}

In \cite{LS}, Lipshitz and Sarkar construct the Khovanov homotopy type of a knot or link $K$, denoted $\chi(K)$.  This homotopy type is a spectrum that satisfies the following properties:

\begin{itemize}
\item $\chi(K)$ is the suspension spectrum of a CW complex.
\item $\tilde{H}^i(\chi(K))=Kh^i(K)$, the (normalized) Khovanov homology of the knot or link $K$.
\item $\chi(K)$ splits as a wedge sum over $q$-degree.  That is, we can write 
\[\chi(K)=\bigvee_{j\in\Z}\chi^j(K)\]
where $\tilde{H}^i(\chi^j(K))=Kh^{i,j}(K)$, the Khovanov homology in $q$-degree $j$.  Any knot or link $K$ has non-zero Khovanov homology in only finitely many $q$-degrees, so this wedge sum is actually finite.
\item Each $\chi^j(K)$ is constructed combinatorially using the Khovanov chain complex $KC^j(K)$ in $q$-degree $j$, together with a choice of ``Ladybug Matching'' that uses the diagram for $K$ (see section 5.4 in \cite{LS}).  In fact, the generators of $KC^j(K)$ correspond precisely to cells in the CW complex which $\chi^j(K)$ is the suspension spectrum of.
\item Each $\chi^j(K)$ is an invariant of the knot or link $K$.  That is to say, the stable homotopy type of $\chi^j(K)$ does not depend on the diagram used to portray $K$, nor on the various choices that are made during the construction.
\end{itemize}

See \cite{LS} for details on the construction, which uses the Khovanov chain complex $KC(K)$ (and the ``Ladybug Matching'' choice) to build a Khovanov flow category that, in a suitable sense, covers the Morse flow category for downward flows between critical points on the vertices of a cube.  The moduli spaces of this Khovanov flow category (which are `manifolds with corners') are embedded in a suitable Euclidean space, and the embedding is framed so that the 0-dimensional moduli spaces (flows between generators on adjacent vertices) are signed according to the sign convention on the Khovanov cube.  This framed embedding is then used to define (large dimensional) cells and attaching maps to build a CW complex whose cellular cochain complex matches $KC(K)$ up to a shift in homological degree.  $\chi(K)$ is the suspension spectrum of this CW complex, suitably de-suspended to account for the homological shift.

We note here that in a subsequent paper (\cite{LS2}), Lipshitz and Sarkar explore the use of cohomology operations on the Khovanov homotopy type (since the spectra are only stable homotopy invariants, there is no well-defined cup product) to detect homotopy types that are not simply wedge sums of Moore spaces, and in \cite{LS3}, to produce new bounds on the slice genus of knots via Rasmussen's $s$-invariant defined in \cite{Ras}.  In particular, the authors derive a purely combinatorial formula for calculating the $\text{Sq}^1$ and $\text{Sq}^2$ Steenrod Square operations.  The simplest knot for which they find such operations to be non-trivial is $T(3,4)$, and indeed the Khovanov homological thickness of torus knots (also explored in \cite{Sto}) makes such knots good candidates for seeking out more nontrivial Steenrod Squares.  Our Corollary \ref{infmanystsquares} gives an example of this.

\subsection{The Collapsing Lemma}

For the convenience of the reader, we restate Lemma 3.32 from \cite{LS} in the form that we will use.

\begin{lemma}
\label{Lemma3.32}
Consider the Khovanov chain complex $KC(K)$ of a knot or link $K$ represented as a chain map:
\begin{equation}
\label{Lemma3.32format}
KC(K) = \big(KC(K'') \longrightarrow KC(K')\big)
\end{equation}
where $K'$ and $K''$ are the links resulting from taking the 1-resolution and 0-resolution, respectively, of a single crossing in the diagram for $K$.  Then if $KC(K')$ is acyclic, we have $\chi(K'')\simeq\chi(K)$ via the induced inclusion map on the spectra.  The same is true when restricting to a specific $q$-degree, so long as the degree shifts are accounted for when viewing $KC(K'')$ and $KC(K')$ as parts of $KC(K)$ on the one hand, or as independent complexes constructed from the diagrams $K''$ and $K'$ on the other.

\end{lemma}
\begin{proof} The assumption implies that $KC(K')$ and $KC(K'')$ correspond to upward and downward closed flow subcategories of the corresponding flow category for $K$, as described in section 3.4.2 in \cite{LS}.  Thus the statement in the lemma is a specific case of Lemma 3.32 in the same paper.  The final point about restricting to a specific $q$-degree is clear since the Khovanov chain complex splits as a direct sum over $q$-degree, and this splitting is respected in the construction of the Khovanov homotopy type which splits as a wedge sum over $q$-degree.

\end{proof}

This lemma allows us to collapse an entire acyclic subcomplex of the Khovanov chain complex in a single $q$-degree without changing the corresponding wedge summand of the Khovanov homotopy type.  The proof of our main result will rely on an inductive application of this idea.


\section{The Setup}

\subsection{Notations and Bounding Lemmas}

We begin this section by collecting the basic notations that will be used throughout the paper, including a few that have already been introduced.

\begin{itemize}
\item $n\in\N$ is fixed, and denotes the number of strands for a given torus link.
\item $\T^m$ will denote the $(n,m)$ torus link viewed as the closure of the $(n,m)$ torus braid on $n$ strands.
\item $a\in\Z$ is fixed, and will denote a fixed $q$-degree.
\item $\chi^j(K)$ will denote the Khovanov homotopy type, in $q$-degree $j$, of the knot or link $K$.
\item $KC^j(K)$ will denote the (normalized) Khovanov chain complex, in $q$-degree $j$, associated to the knot or link $K$.  If the superscript is omitted, the entire Khovanov chain complex over all $q$-degrees is being considered.
\item $\deg_h()$ will denote homological degree, and $\deg_q()$ will denote $q$-degree, of elements in $KC(K)$.  We recall the normalized formulas here:
\begin{equation}
\label{hdeg}
\deg_h()=\#(\text{1-resolutions})-n_-
\end{equation}
\begin{equation}
\label{qdeg}
\deg_q()=\deg_h()+(n_+-n_-)+(\#(v_+)-\#(v_-))
\end{equation}
where $n_+$ and $n_-$ denote the number of positive and negative crossings in the diagram for $K$.  We use $v_+$ and $v_-$ to denote the two standard basis elements in the vector spaces assigned to circles by the Khovanov functor.
\end{itemize}

Our next definition is modelled on the notations of Sto\v{s}i\'{c} in \cite{Sto}.

\begin{definition}
\label{DandE}
For any $m\in\N$, and for any $i\in\{1,2,\dots,n-1\}$, define $D_i^m$ as the link diagram resulting from taking the 0-resolution $\left(\inlinegfx{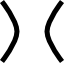}\right)$ of the first $i$ crossings in $\T^{m+1}$, starting from the top left crossing of the braid representation, and working diagonally rightward and downward.  The orientation of $D_i^m$ is induced by the orientation of $\T^{m+1}$.  As a convention, we also define $D_0^m=\T^{m+1}$.  Define $E^m_i$ as the link diagram resulting from taking the 0-resolution of the first $i-1$ crossings in this way, but taking the 1-resolution $\left(\inlinegfx{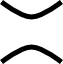}\right)$ of the $i$th crossing.  See the figure below for clarification in the case $n=4$ and $m=2$.  NOTE: We postpone the choice of orientation for $E_i^m$ until after Lemma \ref{StosicLemma1}, which will provide a preferred choice - see Remark (\ref{Eorientation}).
\end{definition}

\begin{figure}[h]
\label{DandEfigure}
\centering
\includegraphics[scale=.2]{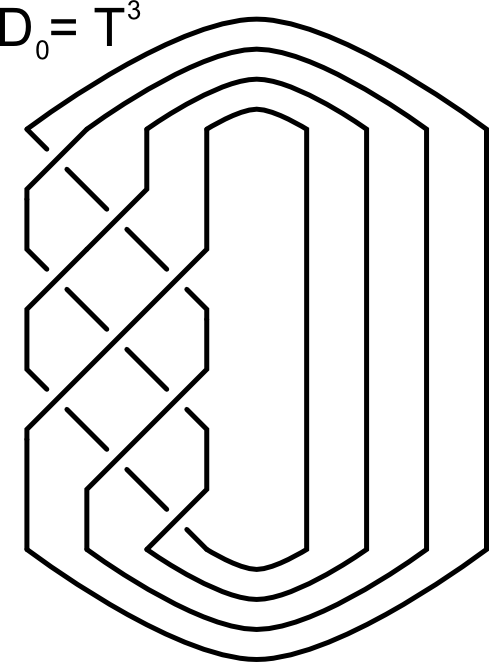}
\\
\vspace{.2in}
\centering
\includegraphics[scale=.2]{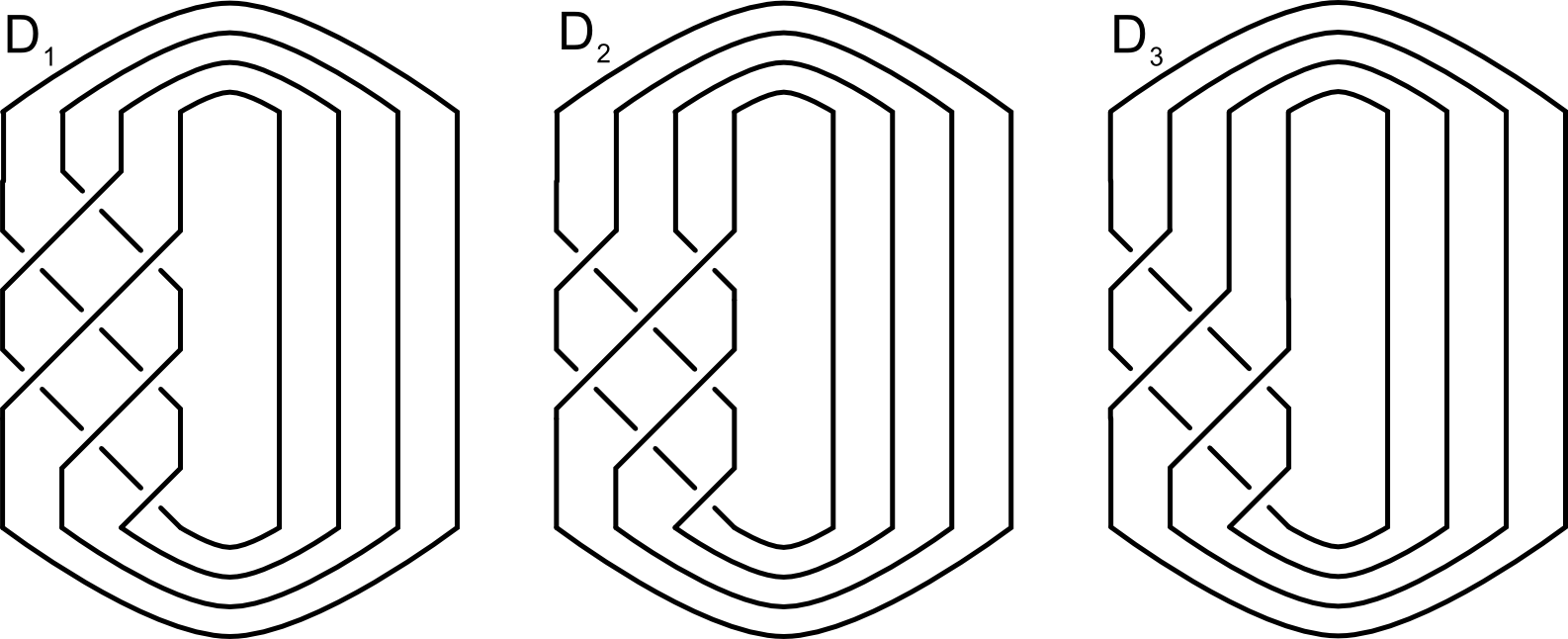}
\\
\vspace{.2in}
\centering
\includegraphics[scale=.2]{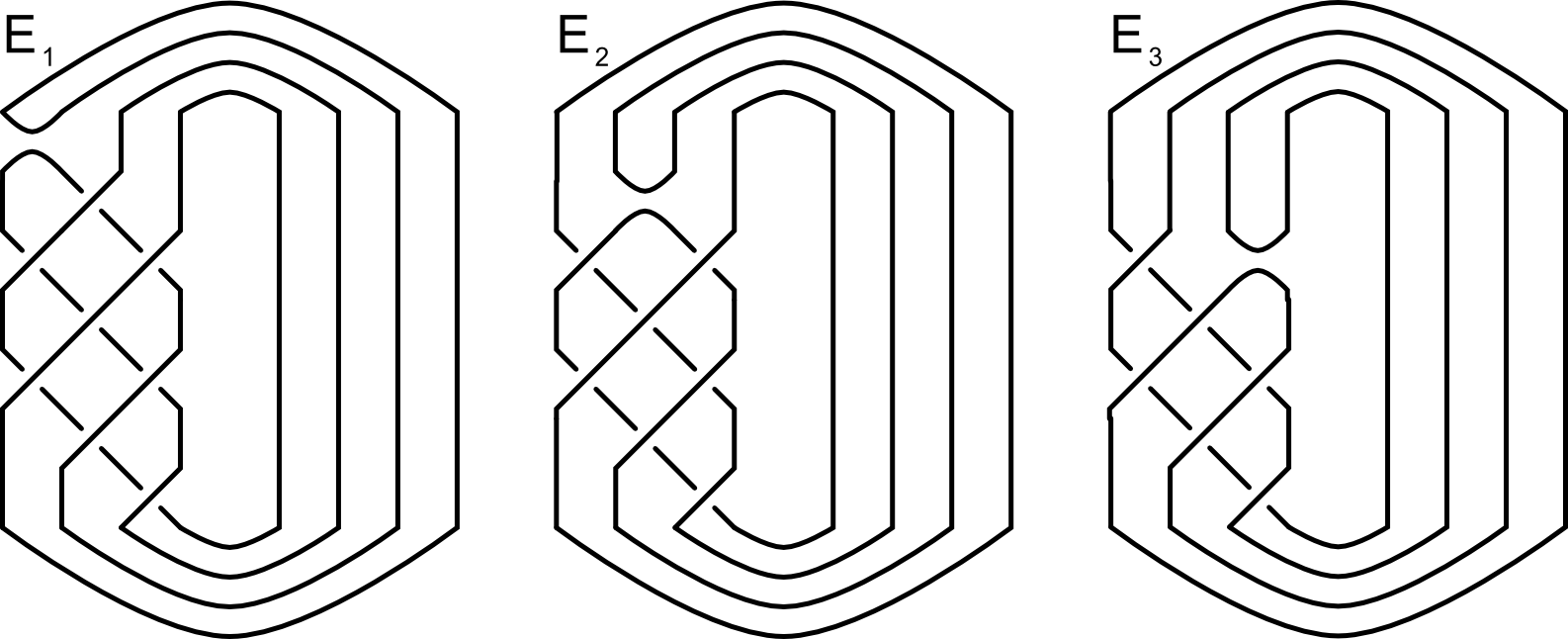}
\caption{Definition \ref{DandE} illustrated for the case $n=4$, $m=2$.  In this case, $D_0=\T^{m+1}=\T^3$, while $D_3=D_{n-1}=\T^m=\T^2$}
\end{figure}

\begin{remark}
\label{Dn-1}
It is clear that, with this notation, $D_{n-1}^m=\T^m$.
\end{remark}

\begin{remark}
\label{DandEomitm}
For the majority of the paper, $m$ will be fixed while handling calculations involving $D_i^m$ and $E_i^m$.  In these cases, the superscript $m$ will be omitted.
\end{remark}

It is clear that in any $D^m_i$ the crossings remain positive.  On the other hand, many of the crossings in $E^m_i$ become negative.  With this in mind, we define:

\begin{definition}
\label{cidefinition}
Let $c^m_i$ denote the number of negative crossings in the diagram $E_i^m$.  As in remark \ref{DandEomitm}, the superscript will usually be omitted when $m$ is fixed as in the majority of the calculations to follow.
\end{definition}

\begin{lemma}[Sto\v{s}i\'{c}]
\label{StosicLemma1}
Fix $m\in\N$ with $m\geq n\geq 3$.  For any $i\in\{1,2,\dots,n-1\}$, the link represented by the diagram $E_i$ is a positive link, and the number of negative crossings satisfies the inequality
\begin{equation}
\label{StosicLemma1eq}
c_i\geq m+n-2
\end{equation}
\end{lemma}
\begin{proof}
This is precisely Lemma 1 in \cite{Sto}, with slightly altered notation.  Thus our $n$ is Sto\v{s}i\'{c}'s $p$, while our $m$ is his $q-1$ (note the bound $m\geq n$, equivalent in $\Z$ to $m+1>n$ which is the bound $q>p$ in \cite{Sto}).  This last shift is due to the fact that Sto\v{s}i\'{c} models his proof (and thus his notation) on the passage from $\T^{m-1}$ to $\T^m$.  In essence, the proof is simply that the top turnback created in $E_i$ can be `swung around' the closure of the braid, then `pulled' through the rest of the diagram, eliminating all of the negative crossings (the turnback may have to be `swung around' multiple times).  A simple count along the way produces the (very crude) bound. See also Remark 4 in the same paper.  Figure 1 makes this intuitively clear.
\end{proof}

\begin{remark}
\label{StosicBoundn3n2}
In the case $n=3$, we will see that this bound can be sharpened to a precise calculation for the $c_i$, while in the case $n=2$, $E_1$ is simply the unknot.
\end{remark}

With Lemma \ref{StosicLemma1} and Remark \ref{StosicBoundn3n2} in mind, we make the following two definitions.

\begin{definition}
\label{Eireduced}
Let $E_{i,red}^m$ denote the diagram resulting from the diagram $E_i^m$ after eliminating all of the negative crossings as in Lemma \ref{StosicLemma1}.  This will be called the \emph{reduced diagram} for the link represented by $E_i^m$.  As above, the superscript will be omitted in most cases when $m$ is being fixed.
\end{definition}

\begin{remark}
\label{Eorientation}
The proof of Lemma \ref{StosicLemma1} guarantees that there is a choice of orientation for $E_i$ so that $E_{i,red}$ is the closure of a positive braid.  We choose this orientation to complete the definition of $E_i$.
\end{remark}

\begin{definition}
\label{xidefinition}
Let $x_i^m$ denote the number of crossings (all positive) in the reduced diagram $E_{i,red}^m$.  Again, the superscript will normally be omitted.
\end{definition}

\begin{lemma}
\label{xibound}
Fix $m\in\N$.  For any $i\in\{1,2,\dots,n-1\}$, the reduced diagram $E_{i,red}$ satisfies the inequality
\begin{equation}
\label{xiboundeq}
x_i\geq m(n-1)-2c_i
\end{equation}
\end{lemma}
\begin{proof}
The diagram $E_i$ is formed starting from $\T^{m+1}$, which has $(m+1)(n-1)$ crossings, and taking $i$ resolutions of crossings.  From there, we remove various crossings via Reidemeister moves as in the proof of Lemma \ref{StosicLemma1}.  Let $y_i$ denote the total number of crossings (positive and negative) removed from $E_i$ in this way, so that
\begin{equation}
\label{xiyi}
x_i=(m+1)(n-1)-i-y_i.
\end{equation}
We claim that
\begin{equation}
\label{yiclaim}
y_i\leq 2c_i
\end{equation}
To see this, we consider how each Reidemeister move affects both sides of the inequality.  Reidemeister 3 moves do not change either $c_i$ or $y_i$.  Any Reidemeister 2 move also clearly maintains equality between the two sides.  A negative Reidemeister 1 move maintains the inequality since it increases $y_i$ by one while increasing $2c_i$ by two.  Meanwhile, a positive Reidemeister 1 move would eliminate a `kink' of the form \inlinegfx{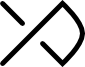} or \inlinegfx{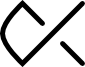}, thus removing the turnback that was being `pulled through' the diagram.  This removal is always the final step in Sto\v{s}i\'{c}'s proof, and as such can be ignored (leaving the positive `kink' in) for the purposes of reducing to a positive diagram.  This proves the claim.

From Equations (\ref{xiyi}) and (\ref{yiclaim}) (and the fact that $i\leq n-1$), we easily obtain
\begin{align*}
x_i&=(m+1)(n-1)-i-y_i\\
&\geq(m+1)(n-1)-(n-1)-2c_i\\
&\geq m(n-1)-2c_i.
\end{align*}
\end{proof}

In order to use Lemma \ref{Lemma3.32}, we will need a bound on the minimal $q$-degree of non-zero Khovanov homology for the subcomplexes corresponding to the $E_i^m$.

\begin{definition}
Let $\min_q(E^m_i)$ denote the minimal $q$-degree of non-zero Khovanov homology for the link represented by $E_i^m$.
\end{definition}

\begin{lemma}
\label{minqbound}
Fix $m\in\Z$ such that $m\geq n\geq 3$.  For any $i\in\{1,2,\dots,n-1\}$, we have
\begin{equation}
\label{minqboundeq}
\min_q(E_i)\geq -n+1+m(n-1)-2c_i
\end{equation}
\end{lemma}
\begin{proof}
First, since Khovanov homology is a graded link invariant, we have that $\min_q(E_i)=\min_q(E_{i,red})$.  Furthermore, the `swinging around' of the top turnback in the proof of Lemma \ref{StosicLemma1} shows that $E_{i,red}$ is the closure of a positive $(\ell,\ell)$-tangle, with $\ell\leq n-2$ (see Figure 1).

Now the minimal $q$-degree of the entire Khovanov chain complex of a link is always achieved by the generator that corresponds to choosing $v_-$ for each circle in the all-zero resolution.  Denote this element $z$.  We have $\deg_q(z)=-$(number of circles) + (normalization).  Each circle must contain at least one max and one min, so in fact 
\[\deg_q(z)\geq-\text{(number of min/max pairs in all-zero resolution) + (normalization)}.\]
For our $E_{i,red}$ we have $\ell$ $(\leq n-2)$ min/max pairs coming from closing the tangle, as well as potentially one more pair coming from the 1-resolution we started with in $E_i$ (no Reidemeister moves in Sto\v{s}i\'{c}'s argument can create more maxs or mins).  The normalization is just the number of positive crossings, that is, $x_i$.  Thus we have
\begin{align*}
\deg_q(z) &\geq -(\ell + 1)+x_i\\
&\geq -((n-2)+1)+(x_i)\\
&\geq -n+1+x_i\\
&\geq -n+1+m(n-1)-2c_i
\end{align*}
where the final line follows from Equation (\ref{xiboundeq}).  Then since $\deg_q(z)$ is as low as possible for the entire Khovanov chain complex, we have $\min_q(E_i)=\min_q(E_{i,red})\geq\deg_q(z)$.
\end{proof}
\begin{remark}
\label{minqboundrmk}
As in Remark \ref{StosicBoundn3n2}, we will see later that the case $n=3$ allows this bound to be sharpened to a precise calculation, while the case $n=2$ gives $\min_q(E_1)=\min_q(U)=-1$.
\end{remark}

\subsection{The Main Strategy}

We now explain the overall strategy in more detail.  Having fixed a specific $q$-degree $a$, we consider $KC^{a+n-1}(\T^{m+1})$ for some value of $m$.  We consider the chain map induced by resolving only the `first' crossing as a 0 or 1 resolution, and so view our cubical complex as
\[KC^{a+n-1}(\T^{m+1}) = \left(KC^{\text{shifted q-degree}}(D_1) \longrightarrow KC^{\text{shifted q-degree}}(E_1)\right).\]
Similarly, we consider the same process applied to $D_1$, $D_2$, and so forth up to $D_{n-2}$.  Recalling the convention that $D_0=\T^{m+1}$, we can write for $i\in\{1,2,\dots,n-1\}$:
\[KC^{\text{shifted q-degree}}(D_{i-1}) = \left(KC^{\text{shifted q-degree}}(D_{i}) \longrightarrow KC^{\text{shifted q-degree}}(E_{i})\right).\]

If we can show that all of the $KC^{\text{shifted q-degree}}(E_i)$ are acyclic, we can use Lemma \ref{Lemma3.32} to collapse them all one at a time, and the result will follow.  To do this, we first consider the shifting of the $q$ degrees.

\begin{definition}
\label{aialphai}
Fix $m\in\N$ and $a\in\Z$.  Define $a_0:=a+n-1$, and then define $a_i$ and $\alpha_i$ inductively so that, for all $i\in\{1,2,\dots,n-1\}$,
\[KC^{a_{i-1}}(D_{i-1}) = \left(KC^{a_i}(D_{i}) \longrightarrow KC^{\alpha_i}(E_{i})\right).\]
\end{definition}
Definition \ref{aialphai} says that any generator $u\in KC^{a+n-1}(\T^{m+1})=KC^{a_0}(D_0)$ that can be viewed as lying in some $KC(D_i)$ will have $\deg_q(u)=a_i$ when viewed as a generator in $KC(D_i)$.  Similarly any generator $w\in KC^{a+n-1}(\T^{m+1})$ that can be viewed as lying in some $KC(E_i)$ will have $\deg_q(w)=\alpha_i$ when viewed as a generator in $KC(E_i)$.

\begin{lemma}
\label{relateaalpha}
Fix $m\in\N$.  For all $i\in\{1,2,\dots,n-1\}$, we have
\begin{equation}
\label{aitoa}
a_i=a+n-1-i
\end{equation}
and
\begin{equation}
\label{alphaitoa}
\alpha_i=a+n-2-i-3c_i
\end{equation}
\end{lemma}

\begin{proof}
Recall Equations (\ref{hdeg}) and (\ref{qdeg}) in the following arguments.  Equation (\ref{aitoa}) is clear, since the diagram $D_i$ keeps all of the remaining crossings positive, and does not shift homological degree.  Thus the only shift comes from the loss of $i$ positive crossings, due to resolving them.

Equation (\ref{alphaitoa}) follows from Equation (\ref{aitoa}) by shifting one additional negative degree due to the homological shift since $E_i$ is formed from a 1-resolution, as well as the normalization shift whereby every positive crossing that switched to a negative crossing shifts the $q$-degree by -3 (-1 for the loss of a positive crossing, and -2 more for the gaining of a negative crossing).
\end{proof}

Before proceeding to prove the main result, we give a detailed example to illustrate all of these notations, the lemma, and the main strategy described above.

\begin{example}
\label{bigexample}
We fix $n=3$ and $m=1$ (so $m+1=2$), the simplest possible case to consider.  We also choose to fix $a=3$, so that $a_0=3+3-1=5$.

We begin with $D_0=\T^2$, the (3,2)-torus knot pictured in Figure 2, with crossings numbered from bottom to top for convenience.

\begin{figure}[h]
\centering
\includegraphics[scale=.28]{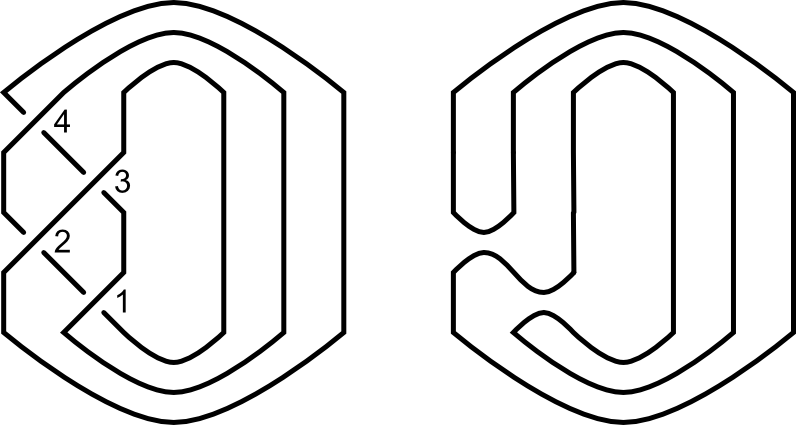}
\caption{On the left is $D_0$, with crossings numbered; on the right is the 1100 resolution}
\end{figure}

We consider the 1100-resolution, also included in the figure.  This resolution clearly gives a single circle, and we let $x$ denote the generator corresponding to $v_-$ assigned to this circle.  We will consider this generator $x$ and its image under various components of the differential.

First, we note that $KC(D_1)$ can be viewed as the portion of $KC(D_0)$ comprising resolutions of the form xxx0, while $KC(E_1)$ is the portion comprising resolutions of the form xxx1 (all up to a shift in $q$-grading, which we will address shortly via $a_1$ and $\alpha_1$).  If we view $x\in \left(KC(D_1)\longrightarrow KC(E_1)\right)$, then its image in $KC(E_1)$ lands in the resolution 1101.  We show the diagrams $D_1$ and $E_1$ in Figure 3, as well as the 1100 and 1101 resolutions.

\begin{figure}[h]
\centering
\includegraphics[scale=.28]{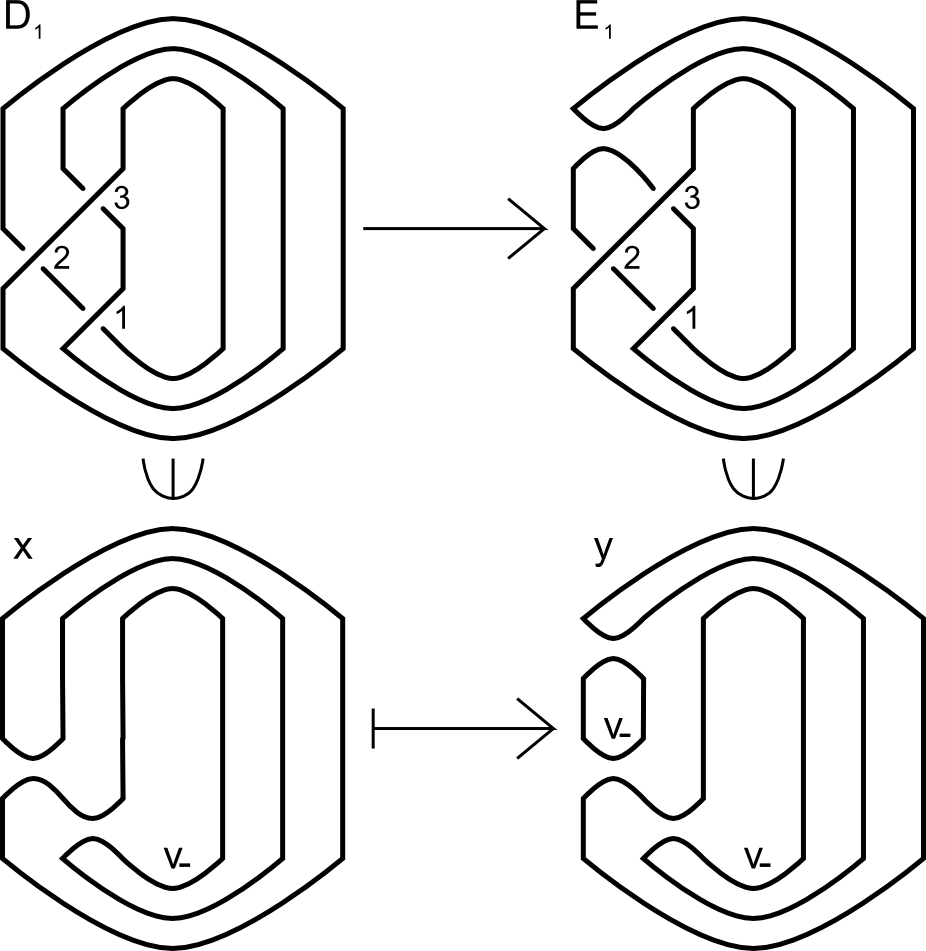}
\caption{The complex viewed diagrammatically as $D_1\longrightarrow E_1$, with the image of $x$ under this component of the differential}
\end{figure}

Recalling the definition of the Khovanov differential, we see that the image of $x$ in this 1101 resolution is the generator $v_-\otimes v_-$.  We denote this generator by $y$, and consider the various grading shifts using Equations (\ref{hdeg}) and (\ref{qdeg}).

First we concentrate on the generator $x$.  If we view $x\in KC(D_0)$, we would consider it as living in the 1100-resolution of a knot diagram with 4 positive crossings and 0 negative crossings, finding $\deg_h(x)=2$ and $\deg_q(x)=2+(4-0)+(0-1)=5$.  Thus, $a_0=5$ (as desired if we wish to fix $a=3$).  Meanwhile, if we view $x\in KC(D_1)$, we would consider it as living in the 110-resolution of a diagram with 3 positive crossings and 0 negative crossings (since the 4th crossing is already resolved), finding $\deg_h(x)=2$ and $\deg_q(x)=2+(3-0)+(0-1)=4$.  Using superscripts to indicate what diagram we are considering, we can summarize as:
\begin{gather*}
\deg_h^{D_0}(x) = 2 = \deg_h^{D_1}(x)\\
a_0=\deg_q^{D_0}(x) = 5\\
a_1=\deg_q^{D_1}(x) = 4
\end{gather*}
Next we concentrate on the generator $y$.  If we view $y\in KC(D_0)$, we would consider it as living in the 1101-resolution of a diagram with 4 positive crossings and 0 negative crossings, finding $\deg_h^{D_0}(y)=3$ and $\deg_q^{D_0}(y)=3+(4-0)+(0-2)=5$, as it should be since the differential viewed entirely in $KC(D_0)$ preserves $q$-grading.  However, if we view $y\in KC(E_1)$, we would consider it as living in the 110-resolution of a diagram with 1 positive crossing and 2 negative crossings, finding $\deg_h^{E_1}(y)=2-2=0$ and $\deg_q^{E_1}(y)=0+(1-2)+(0-2)=-3$.  Summarizing, we have:
\begin{gather*}
\deg_h^{D_0}(y) = 3\\
\deg_h^{E_1}(y) = 0\\
\deg_q^{D_0}(y) = 5\\
\alpha_1=\deg_q^{E_1}(y) = -3
\end{gather*}
The reader can see that this calculation confirms Formula (\ref{alphaitoa}) for $\alpha_1$, noting that here $c_1=2$.  The specific generator $x$ was immaterial; any other generator with same $q$-degree would have followed the same type of calculation.  The key aspect of the strategy outlined at the beginning of this section is that, having fixed $a=3$ so that $a_0=5$, we find $\alpha_1=-3$.  But we can also see that $E_1$ is a diagram for the unknot, which has non-zero homology only in $q$-degree $\pm 1$.  So when $a=3$, the `$E_1$ half' of the complex for $\T^2$ must be acyclic, and can be collapsed.

The next step would be to consider the complex $KC(D_1)$ as $\big(KC(D_2)\longrightarrow KC(E_2)\big)$.  Since $D_2$ is formed by taking the 0-resolution of the top two crossings, we can view $x\in KC(D_2)$ as well.  Its image $w\in KC(E_2)$ lands in the 1110-resolution from the point of view of the diagram $D_0$.  The reader can check that a similar calculation to the one above gives:
\begin{gather*}
\deg_h^{D_0}(x) = 2 = \deg_h^{D_2}(x)\\
a_0=\deg_q^{D_0}(x) = 5\\
a_2=\deg_q^{D_2}(x) = 3 = a \hspace{.1in} (\text{note }n-1=2)\\
\deg_h^{D_0}(w) = 3\\
\deg_h^{E_2}(w) = 1\\
\deg_q^{D_0}(w) = 5\\
\alpha_2=\deg_q^{E_2}(w) = -1\\
c_2=1
\end{gather*}
Once again, the diagram $E_2$ can be seen to be depicting the unknot (this will not be true in full generality, see Lemma \ref{StosicLemma1}).  Unfortunately, the unknot does have non-zero homology in $q$-degree $-1$.  Which means that the portion of $KC^4(D_1)$ corresponding to $KC^{-1}(E_2)$ is not acyclic, and we would not be able to collapse it.  In this example, we are left with no helpful conclusion relating $\chi^5(D_0)=\chi^5(\T^2)$ and $\chi^3(D_2)=\chi^3(\T^1)$.
\end{example}

\section{The Main Result For $n\geq 4$}

In Example \ref{bigexample} at the end of the previous section we came up empty-handed when trying to compare $\chi^{a+n-1}(\T^{m+1})$ to $\chi^a(\T^m)$.  The problem was that, given $a=3$ in that example, our choice of $m=1$ was too small to ensure that all of the $E_i$ could be collapsed.  We now use the formulas and our bounding lemmas produced in Section 3 to state and prove the key theorem which will lead us to the main result.  Throughout this section, $a\in\Z$ will denote a fixed $q$-degree, and $n\in\N$ is fixed with $n\geq 4$.

\begin{theorem}
\label{mainthm}
Fix $a\in\Z$ and $n\in\N$.  Define $f(a,n)=\max(\frac{a+n-1}{n},n)$.  Then for any $m\geq f(a,n)$, and for any $i\in\{1,2,\dots,n-1\}$, we have
\[\chi^{a+n-1}(\T^{m+1})\simeq\chi^{a_i}(D_i).\]
In particular, when $i=n-1$, we have that $a_i=a$ and $D_{n-1}=\T^m$, so that
\[\chi^{a+n-1}(\T^{m+1})\simeq\chi^{a}(\T^m).\]
\end{theorem}
\begin{remark}
The constant term $n$ is present in $f(a,n)$ solely to ensure that the bounding lemmas from section 3 can be used.  This will not be commented upon further within the proof.
\end{remark}
\begin{proof}
We show that, for any $i\in\{1,2,\dots,n-1\}$, we have $\chi^{a_{i-1}}(D_{i-1})\simeq\chi^{a_i}(D_i)$.  The statement of the theorem clearly follows, since $D_0=\T^{m+1}$.  We use the strategy described in the previous section.  That is, we view the Khovanov chain complex for $D_{i-1}$ as
\[KC^{a_{i-1}}(D_{i-1}) = \big(KC^{a_i}(D_i) \longrightarrow KC^{\alpha_i}(E_i)\big)\]
and we set out to prove that, for any such $i$, the subcomplex $KC^{\alpha_i}(E_i)$ is acyclic.  Lemma \ref{Lemma3.32} then gives us the desired stable homotopy equivalence.  To prove that the subcomplex is acyclic, we will show that our assumption on $m$ forces $\alpha_i<\min_q(E_i)$.

We begin from our assumption, fixing $m\geq f(a,n)$ to deduce
\begin{align*}
m &\geq \frac{a+n-1}{n}\\
mn &\geq a+n-1\\
mn-m &\geq a+n-1-m\\
m(n-1) &\geq a+n-1-m+n-n+2-2\\
m(n-1) &\geq a+2n-3-(m+n-2)
\end{align*}
From this we can use Equation (\ref{StosicLemma1eq}) to conclude
\begin{equation}
\label{mainthmeq1}
m(n-1) \geq a+2n-3-c_i.
\end{equation}

Next, we start from our earlier bound (Equation (\ref{minqboundeq})) and apply Equation (\ref{mainthmeq1}) to achieve the desired result:
\begin{align*}
\min_q(E_i) &\geq -n+1+m(n-1)-2c_i\\
&\geq -n+1+a+2n-3-c_i-2c_i\\
&\geq a+n-2-3c_i\\
&> a+n-2-i-3c_i\\
&> \alpha_i.
\end{align*}
Here the last line follows from Equation (\ref{alphaitoa}), while the second to last line gives the strict inequality since $i\geq 1$.
\end{proof}

This theorem shows that, given $q$-degree $a$ and fixed number of strands $n\geq 4$, we can find $M$ (any integer greater than $f(a,n)$) so that $\chi^{a+n-1}(\T^{m+1})\simeq\chi^a(\T^m)$, for any $m\geq M$.  This can be regarded as the $k=1$ case of Theorem \ref{finalresult} (when $n\geq 4$).  In order to complete the proof for all $k$, we need to ensure that the condition $m\geq f(a,n)$ is preserved as $m$ and $a$ grow.

\begin{lemma}
\label{mainlemma}
Using the notations of Theorem \ref{mainthm}, we have:
\begin{equation}
\label{mainlemmaeq}
m\geq f(a,n) \Longrightarrow (m+1)\geq f( (a+n-1), n).
\end{equation}
\end{lemma}
\begin{proof}
Since $f(a,n)$ is a maximum of two quantities, one of which is constant, it is enough to check that
\[m\geq\frac{a+n-1}{n} \Longrightarrow (m+1)\geq\frac{(a+n-1)+n-1}{n}\]
which the reader may easily verify.
\end{proof}

\begin{proof}
(\emph{of Theorem \ref{finalresult} for $n\geq 4$}) We have the case $k=1$ from Theorem \ref{mainthm}.  Lemma \ref{mainlemma} allows us to induct on $k$ and show that, for any $k\in\N$,
\[m\geq f(a,n) \Longrightarrow m+k\geq f(a+k(n-1),n)\]
and thus Theorem \ref{mainthm} allow us to conclude
\[\chi^{a+(k+1)(n-1)}(\T^{m+k+1}) \simeq \chi^{a+k(n-1)}(\T^{m+k}).\]
The result follows.
\end{proof}

\section{The Main Result For $n=2$}
As noted in Remarks \ref{StosicBoundn3n2} and \ref{minqboundrmk}, the bounding lemmas used in the previous section simplify vastly in the case $n=2$.  We continue to use the notations $D_i$ and $E_i$ established in Definition \ref{DandE}, $c_i$ and $x_i$ as in Definitions \ref{cidefinition} and \ref{xidefinition}, and $a_i$ and $\alpha_i$ as in Definition \ref{aialphai}.  However, in the case $n=2$, it is clear that $D_0=\T^{m+1}$ and $D_1=\T^m$, while $E_1$ is a diagram for the unknot with all negative twists (see the figure below).  This allows us to `skip' the bounding lemmas and conclude immediately that, for any fixed $q$-degree $a\in\Z$,
\begin{gather*}
a_0=a+1\\
a_1=a\\
c_1=m\\
x_1=0\\
\min_q(E_1)=-1.
\end{gather*}
\begin{figure}[h]
\centering
\includegraphics[scale=.28]{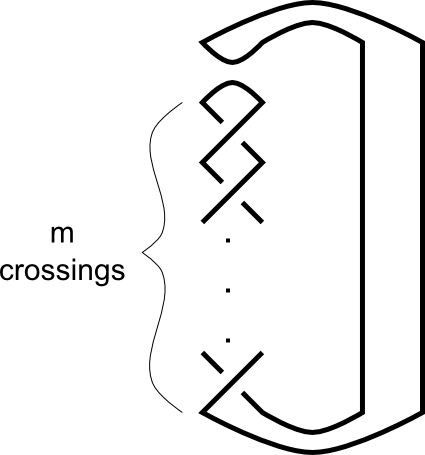}
\caption{The diagram for $E_1$ in the case $n=2$, clearly an unknot}
\end{figure}

We also see that Formula (\ref{alphaitoa}) continues to hold in the case $n=2$, allowing us to conclude
\[\alpha_1=a-3m-1.\]
Thus we have the following theorem:
\begin{theorem}
\label{n=2case}
Fix $a\in\Z$.  For $n=2$ and for any $m>\frac{a}{3}$, we have $\chi^{a+k}(\T^{m+k})\simeq\chi^{a}(\T^m)$ for all $k\in\N$.
\end{theorem}
\begin{proof}
The assumption ensures that $\alpha_1 < a-3(\frac{a}{3})-1 = -1 = \min_q(E_1)$, so that in this $q$-degree the $E_1$ portion of the complex for $\T^{m+1}$ is acyclic and can be collapsed using Lemma \ref{Lemma3.32}.  Furthermore, the inequality $m>\frac{a}{3}$ clearly implies $m+k>\frac{a+k}{3}$, so that the result follows as in the proof of Theorem \ref{finalresult} for $n\geq 4$, at the end of the last section.
\end{proof}

Choosing $M$ to be any integer greater than $\frac{a}{3}$, this proves the $n=2$ case of Theorem \ref{finalresult}.  Note that this is a better bound in this case than the bound $m>\frac{a+n-2}{n}=\frac{a}{2}$ which would have been achieved by attempting to simply use Lemmas \ref{StosicLemma1} and \ref{minqbound} for $n=2$.

\section{The Main Result For $n=3$}

The bounds resulting from Lemmas \ref{StosicLemma1} and \ref{minqbound} can be used to extend Theorem \ref{mainthm} to the case of $n=3$, but just as in the case $n=2$, a better bound can be achieved more directly.  The following lemma provides a more precise version of Remarks \ref{StosicBoundn3n2} and \ref{minqboundrmk}.  Once again we continue to use the notations $D_i$ and $E_i$ established in Definition \ref{DandE}, $c_i$ and $x_i$ as in Definitions \ref{cidefinition} and \ref{xidefinition}, and $a_i$ and $\alpha_i$ as in Definition \ref{aialphai}.  The symbol $U$ will denote the unknot, and $U\sqcup U$ the two component unlink.

\begin{lemma}
\label{Ein3}
When $n=3$, we have the following calculations for $E_i$, $c_i$, and $\alpha_i$ depending on the congruence of $m$ mod 3:
\begin{enumerate}
\item $m=3k$ for some $k\in\N$
\begin{align*}
E_1 &= U & E_2 &= U\sqcup U\\
c_1 &= 4k & c_2 &= 4k\\
x_1 &= 0 & x_2 &= 0\\
\alpha_1 &= a-12k & \alpha_2 &= a-1-12k
\end{align*}
\item $m=3k+1$ for some $k\in\N\cup\{0\}$
\begin{align*}
E_1 &= U & E_2 &= U\\
c_1 &= 4k+2 & c_2 &= 4k+1\\
x_1 &= 0 & x_2 &= 0\\
\alpha_1 &= a-12k-6 & \alpha_2 &= a-12k-4
\end{align*}
\item $m=3k+2$ for some $k\in\N\cup\{0\}$
\begin{align*}
E_1 &= U\sqcup U & E_2 &= U\\
c_1 &= 4k+3 & c_2 &= 4k+3\\
x_1 &= 0 & x_2 &= 0\\
\alpha_1 &= a-12k-9 & \alpha_2 &= a-12k-10
\end{align*}
\end{enumerate}
\end{lemma}
\begin{proof}
Following along Sto\v{s}i\'{c}'s proof of Lemma \ref{StosicLemma1}, the turnback at the top of $E_i$ is swung around to the bottom and then passed through each full twist (corresponding to $m=3$) via two Reidemeister 2 moves and two Reidemeister 1 moves, which together eliminate all of the crossings of the full twist (4 of which were negative).  After this the behavior of the turnback is examined at the `top' of the diagram on a case-by-case basis to determine the results above.  Note that the $\alpha_i$ are determined by the $c_i$ as in Equation (\ref{alphaitoa}).  See also the proof of Theorem 8 in \cite{Sto2}.
\end{proof}

From here, the case $n=3$ proceeds in identical fashion to the other cases.

\begin{theorem}
\label{n=3case}
Fix $a\in\Z$.  For $n=3$ and for any $m\geq\frac{a+2}{4}$, we have $\chi^{a+2k}(\T^{m+k})\simeq\chi^{a}(\T^m)$ for all $k\in\N$.
\end{theorem}
\begin{proof}
Since $\min_q(U)=-1$ and $\min_q(U\sqcup U) = -2$, the reader can check via Lemma \ref{Ein3} that the assumption ensures that $\alpha_i \leq \min_q(E_i)-1$ in all cases.  For example, in the case $m=3k+2$ the assumption gives
\begin{align*}
3k+2 &\geq \frac{a+2}{4}\\
12k+8 &\geq a+2\\
-2 &\geq a-12k-8
\end{align*}
which forces 
\begin{align*}
\alpha_1&=a-12k-9\\
&\leq-3\\
&\leq\min_q(U\sqcup U) -1
\end{align*}
and
\begin{align*}
\alpha_2&=a-12k-10\\
&\leq-4\\
&<\min_q(U)-1.
\end{align*}
Thus in this $q$-degree all of the $E_i$ portions of the relevant complexes are acyclic and can be collapsed using Lemma \ref{Lemma3.32} as in the proof of Theorem \ref{mainthm}.   Furthermore, the inequality $m\geq\frac{a+2}{4}$ clearly implies $m+k\geq\frac{a+2k+2}{4}$, so that the result follows as in the proof of Theorem \ref{finalresult} for $n\geq 4$.
\end{proof}

This concludes the proof of Theorem \ref{finalresult} for all values of $n$.

\section{Defining $\chi(\T^\infty)$ and Proof of Corollary \ref{finalresultcor}}

The stable homotopy equivalences in Theorem \ref{finalresult} can be viewed as the tail of a sequence of morphisms induced by inclusions of subcomplexes as in the proofs of Theorems \ref{mainthm}, \ref{n=2case} and \ref{n=3case}:
\begin{equation}
\label{hocoseq_unnorm}
\chi^{a-m(n-1)}(\T^{0})\longrightarrow\cdots\longrightarrow \chi^{a-(n-1)}(\T^{m-1})\simeq \chi^a(\T^m)\simeq \chi^{a+(n-1)}(\T^{m+1})\simeq \chi^{a+2(n-1)}(\T^{m+2})\simeq\cdots
\end{equation}
If we shuffle the indexing a bit, letting $j=a-m(n-1)$, we see the sequence:
\begin{equation}
\label{hocoseq}
\chi^j(\T^0)\longrightarrow \chi^{j+(n-1)}(\T)\longrightarrow \cdots \longrightarrow \chi^{j+k(n-1)}(\T^k)\longrightarrow\cdots
\end{equation}
where $\T^0$ is taken to be the closure of the identity braid on $n$ strands, that is, the $n$-component unlink.  This allows us to make the following definition.

\begin{definition}
\label{chiTinfdef}
The \emph{limiting Khovanov stable homotopy type of $n$-strand torus links} is denoted $\chi(\T_n^\infty)$ and is defined as a wedge sum
\begin{equation}
\label{infwedge}
\chi(\T_n^\infty)=\bigvee_{j\in\Z}\chi^j(\T_n^\infty)
\end{equation}
where for each \emph{stable $q$-degree} $j\in\Z$,
\begin{equation}
\label{hocodef}
\chi^j(\T_n^\infty) = \text{hocolim}\left(\chi^j(\T_n^0)\longrightarrow \chi^{j+(n-1)}(\T_n)\longrightarrow \cdots \longrightarrow \chi^{j+k(n-1)}(\T_n^k)\longrightarrow\cdots\right).
\end{equation}
\end{definition}
Here we have reintroduced the number of strands into the notation to clarify the formula presented below.  The existence of these hocolims proves Corollary \ref{finalresultcor}.  Theorem \ref{finalresult} ensures that the sequence stabilizes and, since the maps in the sequence are cofibrations (they are induced by inclusions of subcomplexes), we can conclude that the hocolim is stably homotopy equivalent to the first term in the sequence beyond which stabilization begins.  Theorems \ref{mainthm}, \ref{n=2case}, and \ref{n=3case} provide the explicit bounds from which this term can be calculated.  These simple algebraic manipulations are provided in the next three subsections, but here we record the final results.

\begin{corollary}
\label{chiTinfcalc}
Fix $n\geq 4$ in $\N$, and $j\in\Z$.  Then the limiting stable homotopy type $\chi^j(\T^\infty_n)$ satisfies:
\begin{gather*}
\chi^j(\T^\infty_{n\geq 4}) \simeq \begin{cases}
\chi^{nj+(n-1)^2}(\T_n^{j+n-1}) &\text{if } j\geq 1\\
\chi^{j+n(n-1)}(\T_n^n) &\text{if } j<1
\end{cases}\\
\end{gather*}
\end{corollary}

\begin{corollary}
\label{chiTinfcalcn2}
Fix $n=2$, and $j\in\Z$.  Then the limiting stable homotopy type $\chi^j(\T^\infty_2)$ satisfies:
\begin{gather*}
\chi^j(\T^\infty_{n=2}) \simeq \begin{cases}
S^0 &\text{for } j=-2,0\\
S^2 &\text{for } j=2\\
\Sigma^{\frac{j}{2}-1} \R P^2 &\text{for } j\equiv 0 \mod 4, j\geq 4\\
S^\frac{j}{2}\bigvee S^{\frac{j}{2}+1} &\text{for } j\equiv 2 \mod 4, j\geq 6\\
\{*\} &\text{else}
\end{cases}
\end{gather*}
\end{corollary}

\begin{corollary}
\label{chiTinfcalcn3}
Fix $n=3$, and $j\in\Z$.  Then the limiting stable homotopy type $\chi^j(\T^\infty_3)$ satisfies:
\begin{gather*}
\chi^j(\T^\infty_{n=3}) \simeq \begin{cases}
\chi^{2j+3}(\T^{\frac{j+3}{2}}) &\text{for } j \text{ odd}, j\geq -3\\
\{*\} &\text{else}
\end{cases}
\end{gather*}
\end{corollary}

\subsection{Proof of Corollary \ref{chiTinfcalc} ($n\geq 4$)}

Fix $n\geq 4$.  From Theorem \ref{mainthm}, we have that $\chi^{a+k(n-1)}(\T^{m+k})\simeq\chi^a(\T^m)$ so long as $m\geq f(a,n)$.  We view our homotopy types as depending on points in the $(a,m)$-plane, where any line of stably equivalent Khovanov homotopy types will have slope $\frac{1}{n-1}$.  Meanwhile, the lower bound $m=f(a,n)=\max(\frac{a+n-1}{n},n)$ combines the horizontal line $m=n$ (which we denote $L_1$) with the line $m=\frac{1}{n}a+\frac{n-1}{n}$ (which we denote $L_2$) with slope $\frac{1}{n}$.  The situation is shown schematically in Figure 5.

\begin{figure}[h]
\centering
\includegraphics[scale=.5]{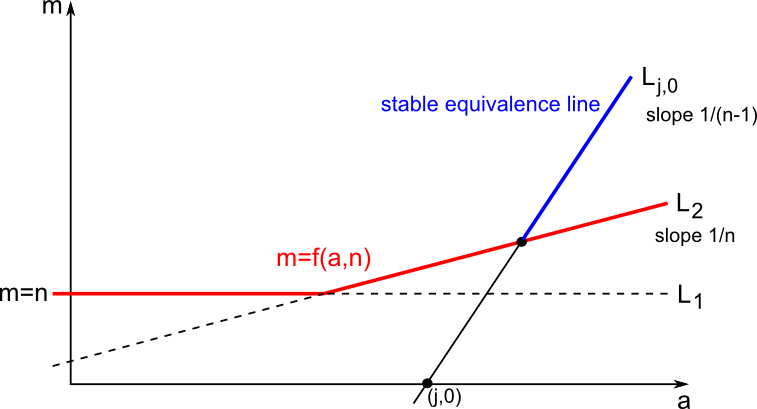}
\caption{Graphing various lines related to $\chi^a(\T^m)$ in the $(a,m)$-plane for fixed $n\geq 4$.  The lower bound $m=f(a,n)$ is shown in red.  A line of stably equivalent Khovanov homotopy types is shown in blue.}
\end{figure}

As indicated in the figure, given the lattice point $(j,0)$ corresponding to $\chi^{j}(\T^{0})$, we draw the line through $(j,0)$ with slope $\frac{1}{n-1}$, which we denote $L_{j,0}$.  The goal is to find the earliest homotopy type after which stabilization begins; that is, we want the intersection point $(\hat{a},\hat{m})=L_{j,0}\cap f(a,n)$ where the blue intersects the red in the figure, giving us the desired $\chi^{\hat{a}}(\T^{\hat{m}})\simeq\chi^{j}(\T^\infty)$.  After some elementary calculations, the piecewise nature of $m=f(a,n)$ results in two cases:

\begin{gather*}
(\hat{a},\hat{m}) = \begin{cases}
(nj+(n-1)^2,j+(n-1)) &\text{for }j\geq 1\\
(j+n(n-1),n) &\text{for }j<1
\end{cases}
\end{gather*}

This proves Corollary \ref{chiTinfcalc}.

\subsection{Proof of Corollary \ref{chiTinfcalcn2} ($n=2$)}

When $n=2$, the picture simplifies considerably because we have only the single linear lower bound $m>\frac{a}{3}$.  If we denote the line $m=\frac{a}{3}$ by $L_1$, and continue to use the notation $L_{j,0}$ for our line passing through $(j,0)$ (now having slope 1), we seek the intersection $(\hat{a},\hat{m})=L_{j,0}\cap L_1$.  See Figure 6.

\begin{figure}[h]
\centering
\includegraphics[scale=.5]{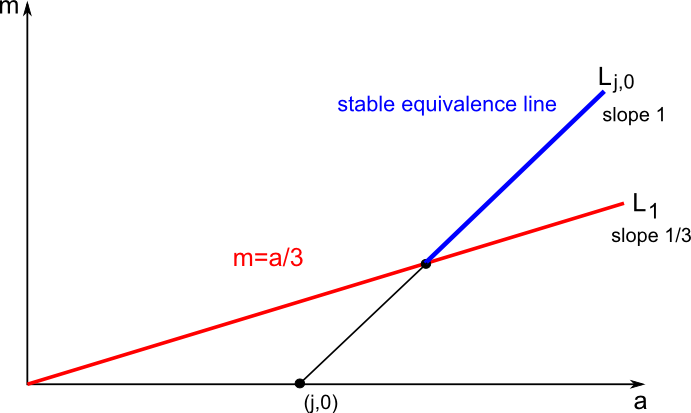}
\caption{Graphing various lines related to $\chi^a(\T^m)$ in the $(a,m)$-plane for $n=2$.  The lower bound $m=\frac{a}{3}$ is shown in red.  A line of stably equivalent Khovanov homotopy types is shown in blue.}
\end{figure}

The intersection point $L_{j,0}\cap L_1$ is quickly found to be the point $\left(\frac{3j}{2},\frac{j}{2}\right)$.  However, our lower bound in this case is a strict inequality, and so we choose $\hat{m}$ (and then $\hat{a}$) as follows:
\begin{equation}
\label{n=2mhat}
\hat{m}=\left\lceil \frac{j+1}{2} \right\rceil
\end{equation}
\begin{equation}
\label{n=2ahat}
\hat{a}=\hat{m}+j
\end{equation}
where $\lceil\cdot\rceil$ indicates the ceiling function.  Thus we have
\begin{equation}
\label{n=2firststep}
\chi^j(\T^\infty_{n=2}) \simeq \chi^{j+\left\lceil \frac{j+1}{2} \right\rceil}(\T^{\left\lceil \frac{j+1}{2} \right\rceil})
\end{equation}
so long as $\left\lceil \frac{j+1}{2} \right\rceil\geq 0$.  For smaller $j$ we have simply $\chi^j(\T^0)=\chi^j(U\sqcup U)$, which is clearly a point for $j<-2$.  From here we use facts about the Khovanov homology of $(2,m)$ torus links to simplify our expressions.

\begin{lemma}
\label{n=2facts}
The knots and links $\T_2^m$ satisfy the following properties:
\begin{enumerate}
\item $\T_2^m$ is an alternating link, and therefore its Khovanov homotopy type is a wedge sum of Moore spaces
\item $Kh^{h,q}(\T^m_2)$ is empty if the parity of $q$ and $m$ do not match (here $h$ indicates homological degree, $q$ indicates $q$ degree)
\item For each odd $m\geq 3$, 
\begin{gather*}
Kh^{h,3m-2}(\T^m_2)\cong \begin{cases}
\Z_2 &\text{for } h=m\\
0 &\text{else}
\end{cases}
\end{gather*}
\item For each even $m\geq 4$,
\begin{gather*}
Kh^{h,3m-2}(\T^m_2)\cong \begin{cases}
\Z &\text{for } h=m-1,m\\
0 &\text{else}
\end{cases}
\end{gather*}
\end{enumerate}
\end{lemma}

\begin{proof}
The simplicity of the Khovanov homotopy type for alternating links is proved in section 9.3 of \cite{LS} as a consequence of their Khovanov homology being homologically thin.  The rest of the facts here are simple calculations and are left to the reader (see also section 4.3.1 in \cite{CK})
\end{proof}

Returning to our calculation, we find that for odd $j$, the ceiling functions can be ignored and Equation (\ref{n=2mhat}) implies $j=2\hat{m}-1$.  This allows us to write, for $j$ odd,
\[\chi^j(\T_2^\infty) \simeq \chi^{3\hat{m}-1}(\T^{\hat{m}})\]
and since the parity of $3\hat{m}-1$ and $\hat{m}$ are opposite, Lemma \ref{n=2facts} shows that the Khovanov homotopy type in this case is a point.  For $j\cong 0\mod 4$ we find $\hat{m}$ even and $\hat{a}=3\hat{m}-2$, while for $j\cong 2 \mod 4$ we find $\hat{m}$ odd and again $\hat{a}=3\hat{m}-2$, allowing us to use Lemma \ref{n=2facts} to build the Khovanov homotopy types of such cases as wedge sums of Moore spaces $\R P^2$ and $S^0$ suspended appropriately to sit in the correct homological grading.  In both cases, we have $j=2\hat{m}-2$ and, after suitable algebraic manipulations of the indices, Equation (\ref{n=2firststep}) gives the desired result.  The final case is $\chi^2(\T_2^\infty)$ which Equation (\ref{n=2firststep}) gives as $\chi^4(\T^2_2)$ which is calculated in section 9 of \cite{LS} to be $S^2$.  This concludes the proof of Corollary \ref{chiTinfcalcn2}.

\subsection{Proof of Corollaries \ref{chiTinfcalcn3} and \ref{infmanystsquares} ($n=3$)}

For the case $n=3$ the picture is essentially the same as in the $n=2$ case (Figure 6).  Our line of stably equivalent homotopy types $L_{j,0}$ now has slope $\frac{1}{2}$, while the bounding line $L_1$ has equation $m=\frac{a+2}{4}$ and slope $\frac{1}{4}$.  The intersection is $L_{j,0}\cap L_1 = (2+2j,\frac{j+2}{2})$.  In this case if $j$ is even, we can use this point; if $j$ is odd, we find the next $\Z^2$ lattice point above this along $L_{j,0}$ to learn
\begin{gather}
\label{n=3firststep}
\chi^j(\T_{n=3}^\infty) \simeq \begin{cases}
\chi^{2j+2}(\T_3^{\frac{j+2}{2}}) &\text{for } j\geq -2 \text{ even}\\
\chi^{2j+3}(\T_3^{\frac{j+3}{2}}) &\text{for } j\geq -3 \text{ odd}\\
\chi^j(\T_3^0) &\text{for } j<-3
\end{cases}.
\end{gather}
The final case is simply to ensure that $m\geq 0$, and will be rendered moot shortly.

As in the case of $n=2$, we proceed by using some simple facts about the Khovanov homology of $(3,m)$ torus links.  First we note that, for $\T_3^0=U\sqcup U\sqcup U$, the lowest $q$-degree available for the Khovanov chain complexes is precisely $-3$, and so the final case in Equation (\ref{n=3firststep}) produces no non-trivial Khovanov homotopy types.  Second, we prove the following lemma.

\begin{lemma}
\label{n=3noevens}
The Khovanov chain complex $KC(\T_3^m)$ is empty in even $q$ degrees for any $m$.
\end{lemma}
\begin{proof} Regardless of $m$, the all 0-resolution will consist of 3 circles and will be in homological degree 0 since all of the crossings were positive.  Furthermore, there must be an even number of crossings since $n_+=(3-1)m$, $n_-=0$.  Thus any generator $x\in KC(\T_3^m)$ in this resolution must have
\begin{align*}
\deg_q(x) &= \deg_h(x)+n_+-n_-+(\#v_+-\#v_-)\\
&= 0+2m+0+(\#v_+-\#v_-)
\end{align*}
which must clearly be odd for 3 circles.  The parity of $q$-degree for the Khovanov chain complex is constant throughout all homological degrees, proving the claim.
\end{proof}

Combining Equation (\ref{n=3firststep}) with Lemma \ref{n=3noevens} and the remarks shortly beforehand completes the proof of Corollary \ref{chiTinfcalcn3}.

\begin{proof}[Proof of Corollary \ref{infmanystsquares}]
From Theorem \ref{n=3case} and Corollary \ref{chiTinfcalcn3}, we have in particular for $j=3$,
\[\chi^3(\T_3^\infty)\simeq\chi^9(\T_3^3)\simeq\chi^{11}(\T_3^4)\simeq\chi^{13}(\T_3^5)\simeq\cdots\]
and in \cite{LS2}, a non-trivial Steenrod Square $Sq^2$ operation is to used to calculate $\chi^{11}(\T_3^4)\simeq \Sigma^{-1}(\R P^5 / \R P^2 )$.  Thus all of the $\chi^{2m+3}(\T_3^m)$ beyond this point have the same Khovanov homotopy type with non-trivial $Sq^2$.
\end{proof}

\bibliographystyle{alpha}

\bibliography{Khovanov_Homotopy_Type_Stabilizes}

\begin{thebibliography}{GORS14}

\bibitem[CK12]{CK}
Benjamin Cooper and Vyacheslav Krushkal.
\newblock Categorification of the {J}ones-{W}enzl projectors.
\newblock {\em Quantum Topol.}, 3(2):139--180, 2012.

\bibitem[GOR13]{GOR}
Eugene Gorsky, Alexei Oblomkov, and Jacob Rasmussen.
\newblock On stable {K}hovanov homology of torus knots.
\newblock {\em Exp. Math.}, 22(3):265--281, 2013.

\bibitem[GORS14]{GORS}
Eugene Gorsky, Alexei Oblomkov, Jacob Rasmussen, and Vivek Shende.
\newblock Torus knots and the rational {DAHA}.
\newblock {\em Duke Math. J.}, 163(14):2709--2794, 2014.

\bibitem[Hog]{Hog}
Matthew Hogancamp.
\newblock Stable homology of torus links via categorified {Y}oung symmetrizers
  {I}: one-row partitions.
\newblock arXiv:1505.08148.

\bibitem[Kho00]{Khov}
Mikhail Khovanov.
\newblock A categorification of the {J}ones polynomial.
\newblock {\em Duke Math. J.}, 101(3):359--426, 2000.

\bibitem[LLS]{LLS}
Tyler Lawson, Robert Lipshitz, and Sucharit Sarkar.
\newblock The cube and the {B}urnside category.
\newblock arXiv:1505.00512.

\bibitem[LS14a]{LS}
Robert Lipshitz and Sucharit Sarkar.
\newblock A {K}hovanov stable homotopy type.
\newblock {\em J. Amer. Math. Soc.}, 27(4):983--1042, 2014.

\bibitem[LS14b]{LS3}
Robert Lipshitz and Sucharit Sarkar.
\newblock A refinement of {R}asmussen's {$S$}-invariant.
\newblock {\em Duke Math. J.}, 163(5):923--952, 2014.

\bibitem[LS14c]{LS2}
Robert Lipshitz and Sucharit Sarkar.
\newblock A {S}teenrod square on {K}hovanov homology.
\newblock {\em J. Topol.}, 7(3):817--848, 2014.

\bibitem[Ras10]{Ras}
Jacob Rasmussen.
\newblock Khovanov homology and the slice genus.
\newblock {\em Invent. Math.}, 182(2):419--447, 2010.

\bibitem[Roz14]{Roz}
Lev Rozansky.
\newblock An infinite torus braid yields a categorified {J}ones-{W}enzl
  projector.
\newblock {\em Fund. Math.}, 225(1):305--326, 2014.

\bibitem[Sto07]{Sto}
Marko Sto{\v{s}}i{\'c}.
\newblock Homological thickness and stability of torus knots.
\newblock {\em Algebr. Geom. Topol.}, 7:261--284, 2007.

\bibitem[Sto09]{Sto2}
Marko Sto{\v{s}}i{\'c}.
\newblock Khovanov homology of torus links.
\newblock {\em Topology Appl.}, 156(3):533--541, 2009.

\end{thebibliography}


\end{document}